\documentclass{amsart}

\usepackage{amscd,amsmath,amssymb,mathrsfs,latexsym}
\usepackage[pdftex]{graphicx}
\usepackage[mathcal]{euscript}
\usepackage[all]{xy}
\usepackage{color}
\usepackage{fancyhdr}
\usepackage{color}
\usepackage{graphicx}
\usepackage{caption}
\usepackage{subcaption}
\usepackage{color,soul} 
\usepackage{hyperref} 
\hypersetup{
 colorlinks=true, 	 
 linktoc=all,     	 
 linkcolor=black,  	 
 citecolor=black,    
 filecolor=magenta,  
 urlcolor=cyan       
}
\usepackage{tikz} 		
\usepackage{tikz-cd}
\usetikzlibrary{matrix,arrows,decorations.pathmorphing,intersections}

\usepackage{pifont} 
\usepackage{fourier-orns} 

\newtheorem{thm}{Theorem}[section]

\newtheorem{lemma}[thm]{Lemma}
\newtheorem{prop}[thm]{Proposition}
\newtheorem{cor}[thm]{Corollary}

\newtheorem{defn}[thm]{Definition}

\newtheorem{ex}[thm]{Example}
\newtheorem{rmk}{Remark}

\newcommand{\Z}{{\mathbb{Z}}}
\newcommand{\R}{{\mathbb{R}}}
\newcommand{\C}{{\mathbb{C}}}
\newcommand{\Q}{{\mathbb{Q}}}
\newcommand{\K}{{\mathbf{k}}}

\newcommand{\T}{{\mathcal{T}}}
\newcommand{\Ho}{{\rm Hom}}

\newcommand{\colim}{{{\rm colim}\hspace{.2em}}}

\newcommand{\Sp}{{{\rm Sp}}}
\newcommand{\Hilb}{{{\rm Hilb}}}
\newcommand{\Trace}{{{\rm Trace}}}

\newcommand{\RR}{{\mathbb {R}}}
\newcommand{\NN}{{\mathbb {N}}}
\newcommand{\TTT}{{\mathcal{T}}}

\newcommand{\barE}{{\widetilde{E}}}

\def\ds{\displaystyle}

\title{On spaces of commuting elements in Lie groups}

\author{Frederick R. Cohen}
\address{Department of Mathematics, University of Rochester, Rochester, NY 14627}
\email{fred.cohen@rochester.edu}

\address{School of Mathematics, University of Minnesota, Minneapolis, MN 55455}
\email{reiner@math.umn.edu}

\author{Mentor Stafa}
\address{Department of Mathematics, Tulane University, New Orleans, LA 70118}
\email{mstafa@tulane.edu}

\date{\today}

\dedicatory{\sc with an appendix by vic reiner}


\thanks{The first author partially supported by the Institute for Mathematics and its Applications.}
\thanks{Second author partially supported by DARPA grant number N66001-11-1-4132 during 2013-2015, and Swiss Government Excellence Scholarship (ESKAS No. 2015.0182) during 2015-2016.}
\thanks{Current address of second author: Departement Mathematik, ETH Z\"urich, 8092, Switzerland.}

\begin{document}

\begin{abstract}
The main purpose of this paper is to introduce a method to ``stabilize'' certain spaces of homomorphisms from finitely generated free abelian groups to a Lie group $G$, namely $\Ho(\mathbb Z^n,G)$. 
We show that this stabilized space of homomorphisms decomposes after suspending once with ``summands'' which can be reassembled, in a sense to be made precise below, into the individual spaces  $\Ho(\mathbb Z^n,G)$ after suspending once. To prove this decomposition,  a stable decomposition of an equivariant function space is also developed. 
One main result is that the topological space of all commuting elements in a compact 
Lie group is homotopy equivalent to an equivariant function space after inverting the order
of the Weyl group.
In addition, the homology of the stabilized space admits a very simple description in terms of the tensor algebra generated by the reduced homology of a maximal torus in favorable cases. The stabilized space also allows the description of the additive reduced homology of the individual spaces $\Ho(\mathbb Z^n,G)$, with the order of the Weyl group inverted.
\end{abstract}

\maketitle

\tableofcontents

\section{Introduction}\label{section:Introduction}

In this paper we introduce a new method of ``stabilizing'' spaces of homomorphisms $\Ho(\pi,G)$, where $\pi$ is a certain choice of finitely generated discrete group and $G$ is a compact and connected Lie group. The main results apply to $\Ho(\Z^n,G)$, the space of all ordered $n$-tuples of pairwise commuting elements in a compact Lie group $G$, by assembling these spaces into a single space for all $n \geq 0$.

The resulting space, denoted $Comm(G)$, is an infinite dimensional analogue of a Stiefel manifold which can be regarded as the space, suitably topologized, of all finite ordered sets of generators for all finitely generated abelian subgroups of $G$.
The methods are to develop the geometry and topology of the free associative monoid generated by a maximal torus of $G$, and to ``twist'' this  free monoid into a space which approximates the space of ``all commuting $n$-tuples'' for all $n$ into a single space. 
Topological properties of $Comm(G)$ are developed while the singular homology of this space is computed with coefficients in the ring of integers with the order of the Weyl group of $G$ inverted. One application is that the cohomology of $\Ho(\Z^n,G)$ follows from that of $Comm(G)$ for any cohomology theory.

The space of all commuting elements is usually homotopy equivalent to the equivariant function space $G \times_{NT} J(T)$ after inverting the order of the Weyl group, where $NT$ is the normalizer of a maximal torus $T$, and $J(T)$ is the free associative monoid generated by the maximal torus.

The results for singular homology of $Comm(G)$ are given in terms of the tensor algebra generated by the reduced homology of a maximal torus. Applications to classical Lie groups as well as exceptional Lie groups are given.  A stable decomposition of $Comm(G)$ is also given here with a significantly finer stable decomposition to be given in the sequel to this paper along with extensions of these constructions to additional representation varieties.

An appendix  by V.~Reiner is included which uses the results here concerning $Comm(G)$ together with Molien's theorem to give the Hilbert-Poincar\'e series of $Comm(G)$.

The next section provides a detailed list of the results in this paper.

\section{Summary of results }\label{section: Outline of results}

Let $G$ be a Lie group and let $\pi$ be a finitely generated discrete group. The set of homomorphisms $\Ho(\pi,G)$ can be topologized with the subspace topology of $G^m$ where $m$ is the number of generators of $\pi$.  The topology of the spaces $\Ho(\pi,G)$ has seen considerable recent development. For example, the case where $\pi$ is a finitely generated abelian group is closely connected to 
work of E. Witten \cite{witten1,witten2} which uses commuting pairs to construct quantum vacuum states in supersymmetric Yang-Mills theory. Further work was done by V.~Kac and A.~Smilga \cite{kac.smilga}.

Work of A.~Borel, R.~Friedman and J.~Morgan \cite{borel2002almost} addressed the special cases of commuting pairs and triples in compact Lie groups. Spaces of representations were studied by W.~Goldman \cite{goldman}, who investigated their connected components, for $\pi$ the fundamental group of a closed oriented surface and $G$ a finite cover of a projective special linear group.

For non-negative integers $n$, let $\Ho(\Z^n,G)$ denote the set of homomorphisms from a direct sum of $n$ copies of $\mathbb{Z}$ to $G$. This set can be identified as the subset of pairwise commuting $n$-tuples in $G^n$, so similarly it can be naturally topologized with the subspace topology in $G^n$.  Spaces given by $\Ho(\Z^n,G)$  have been the subject of substantial recent work.

In particular, A. \'Adem and F.~Cohen \cite{adem2007commuting} first studied these spaces in general, obtaining results for closed subgroups of $GL_n(\C)$. Their work was followed by work of T.~Baird \cite{bairdcohomology}, who studied ordinary and equivariant cohomology, Baird--Jeffrey--Selick \cite{bairdsu2}, \'Adem--G\'omez \cite{adem.gomez,adem.gomez2,adem.gomez3}, \'Adem--Cohen--G\'omez \cite{adem.cohen.gomez,adem.cohen.gomez2}, Sjerve--Torres-Giese \cite{giese.sjerve}, \'Adem--Cohen--Torres-Giese \cite{fredb2g}, Pettet--Suoto \cite{pettet.souto}, G\'omez--Pettet--Suoto \cite{gomez.pettet.souto}, Okay \cite{okay}. Most of this work has been focused on the study of invariants such as cohomology, $K$-theory, connected components, homotopy type and stable decompositions. Recently, D.~Ramras and S. Lawton \cite{ramras} used some of the above work to study character varieties.

Let $G$ be a reductive algebraic group and let $K \subseteq G$ be a maximal compact subgroup. A.~Pettet and J.~Souto \cite{pettet.souto} have shown that the inclusion $\Ho(\Z^n,K) \hookrightarrow \Ho(\Z^n,G)$ is a strong deformation retract, i.e. in particular the two spaces are homotopy equivalent. Thus this article will restrict to compact and connected Lie groups. Note that the free abelian group can also be replaced by any finitely generated discrete group. Assume $\Gamma$ is a finitely generated nilpotent group. Then M.~Bergeron \cite{bergeron} showed that the natural map
$$\Ho(\Gamma,K) \to \Ho(\Gamma,G)$$  is a homotopy equivalence. However, most of the results in this paper apply only to the cases of free abelian groups of finite rank.

The varieties of pairwise commuting $n$-tuples can be assembled into a useful single space which is roughly analogous to a Stiefel variety described as follows. Recall that the Stiefel variety over a field $\K$ $$V_{\K}(n,m)$$ may be regarded as a topological space of ordered $m$-tuples of generators for every $m$ dimensional vector subspace of $\K^n$, see \cite{hatcher2002algebraic}.

The purpose of this paper is to study an analogue of a Stiefel 
variety where $\K^n$ is replaced by a 
compact Lie group $G$, and the $m$-dimensional subspaces are replaced by a particular family of subgroups such as 
\begin{enumerate}
\item all finitely generated subgroups with at most $m$ generators which gives rise to  $\Ho(F_m,G)$, where $F_m$ denotes the free group with a basis of $m$ elements, or
\item all finitely generated abelian subgroups with at most $m$ generators, which gives rise to  $\Ho(\Z^m,G)$.

\end{enumerate}

The analogue of $(1)$ where $\pi$ runs over all finitely generated subgroups of nilpotence class at most $q$  with at most $k$ generators, or the analogue of $(2)$ where $\pi$ runs all finitely generated elementary abelian $p$-groups with at most $k$ generators will be addressed elsewhere. This article will focus mainly on properties of examples (1) and (2) as well as how these spaces make contact with classical representation theory.

\begin{defn}\label{definition: spaces of all commuting n-tuples}
Given a Lie group $G$ together with a free group $F_k$ on $k$ generators, the space of all homomorphisms $\Ho(F_k,G)$ is naturally homeomorphic to $G^k$ with the subspace $\Ho(\Z^k,G)$ topologized by the
subspace topology. Define $Assoc(G)$, $Comm(G)$, and $Comm(q,G)$ by the following constructions.
\begin{enumerate}
\item Let  $$Assoc(G) =  \bigsqcup_{0 \leq k < \infty}\Ho(F_k,G))/\sim$$ where $\sim$
denotes the equivalence relation obtained by deleting the identity element of $G$ 
in any coordinate of  $\Ho(F_k,G)$ for $ k \geq 1$ with $\Ho(F_0,G) = \{1_G\}$ a single point by convention. In addition, $Assoc(G)$ is given the
usual quotient topology obtained from the product topology on  $\Ho(F_k,G)$
where the relation is generated by requiring that 
$$(g_1,\dots, g_n) \sim (g_1,\dots, \widehat{g_i}, \dots, g_n)$$ 
if $g_i = 1_G \in G$.

\item Let $$Comm(G) =   \bigsqcup_{0 \leq k < \infty}\Ho(\Z^k,G))/\sim$$
topologized as a subspace of $Assoc(G)$, where $\sim$ is the same relation.

\item Let $$Comm(q,G) =   \bigsqcup_{0 \leq k < \infty}\Ho(F_k/\Gamma^q(F_k),G))/\sim$$
topologized as a subspace of $Assoc(G)$, where $\sim$ is the same relation, and $\Gamma^q(F_k)$
denotes the $q$-th stage of the descending central series for $F_k$.

\end{enumerate}

\end{defn}

This article is a study of the properties of the space $\Ho(\mathbb{Z}^n,G)$, for certain Lie groups and for all positive integers $n$ assembled into a single space, and is an exploration of properties of $Comm(G)$ regarded as a subspace of $Assoc(G)$. The main reason for introducing the spaces $Comm(G)$ is that (1) they have tractable, natural properties which (2) descend to give tractable properties of the varieties $\Ho(\mathbb{Z}^n,G)$ as addressed in this article.

Consider the free associative monoid generated by $G$ denoted $J(G)$, with the identity element in $G$ identified with the identity element in the free monoid  $J(G).$  This last construction is also known as the the \textit{James construction} or \textit{James reduced product} developed in \cite{james1955reduced} (see also Definition \ref{jamesdefn}). The first theorem is an observation which provides an identification of $Assoc(G)$ as well as a framework for $Comm(G)$; the proof is omitted as it is a natural inspection.

\begin{thm}\label{theorem: Assoc(G)}
Let $G$ be a  Lie group. The natural  map $$J(G) \to Assoc(G)$$ is a homeomorphism.
\end{thm}

One application of the current paper is to give explicit features of the singular cohomology or homology of
$Comm(G)$ for any compact and connected Lie group $G$ with the order of the Weyl group inverted.
For example, one of the results below gives that the cohomology of $Comm(G)$ with the order of the Weyl group inverted is  the fixed points of the Weyl group $W$ acting naturally on $H^*(G/T) \otimes \T^*[V]$, where $\T^*[V]$ denotes the dual of the tensor algebra $\T[V]$ generated by the reduced homology of a maximal torus. 
Furthermore, if ungraded singular homology is considered, the total answer in singular homology is given explicitly in terms of classical tensor algebras generated by the reduced homology of a maximal torus.

The singular homology groups of each individual space $\Ho(\mathbb{Z}^n,G)$  are then formally given in terms of those of $Comm(G)$ as described below. To give the explicit answers, a tighter hold on fixed points of the action of $W$ on
$H^*(G/T) \otimes \T^*[V]$ 
is required.  In particular, Molien's theorem can be applied to work out the fixed points in $H^*(G/T;\C) \otimes \T^*[V \otimes_{\Z} \C]$, the resulting cohomology groups with complex coefficients $\C$ as given in the appendix.

In the ungraded setting, the homology groups obtained from classical representation theory simplify greatly, and the ungraded answers follow directly. This simplification is illustrated by the explicit answers for the cases of $U(n)$, $SU(n)$, $Sp(n)$, $Spin(n)$, $G_2, F_4, E_6, E_7, E_8$, reflecting the fact that these global computations for $Comm(G)$ are both accessible, and are exhibiting natural combinatorics. These specific answers are given in sections 
\ref{section:Un.SUn}, and \ref{section:exceptional Lie groups}.

This information depends on the construction given next.

\begin{defn}\label{definition: twist J(T)}
Given a Lie group $G$ together with a maximal torus $T\subset G$, form the free associative monoid generated by $T$, denoted $J(T) = Assoc(T)$, where the identity of $T$ is identified with the identity of $Assoc(T)$. The normalizer of $T$ in $G$, denoted $NT$, acts on $T$, and by natural diagonal extension there is an action of $NT$ on $Assoc(T)$. The normalizer $NT$ thus acts diagonally on the product $G \times Assoc(T)$. The orbit space 
$$G \times_{NT}Assoc(T)$$ is useful in what follows next.
\end{defn}

\begin{defn}\label{definition: G to Comm(G)}
Given a Lie group $G$, define  $$E:G \to Comm(G)$$ by 
$$E(g)=[g],$$ 
the natural equivalence class given by the image of $g \in G  = \Ho(\Z,G)$ in $Comm(G)$. Let 
$$Comm(G)_{1_G}$$ 
denote the path-component of $E(1_G) \in Comm(G)$. Let 
$$\Ho(\Z^m,G)_{1_G} $$ 
denote the path-component of $({1_G}, \dots, 1_G) \in \Ho(\Z^m,G) $.
\end{defn}

The next Proposition follows at once from \cite{adem2007commuting} and Definition \ref{definition: spaces of all commuting n-tuples}.
\begin{prop}\label{prop: connected Comm(G)}
Let $G$ be a  Lie group. If every abelian subgroup of $G$ is contained in a maximal torus, then
both  $\Ho(\mathbb{Z}^n,G)$, and $Comm(G)$ are path-connected.
\end{prop}

Standard examples of Proposition \ref{prop: connected Comm(G)} are $U(n)$, $SU(n)$, and $Sp(n)$.
On the other hand if $G = SO(n)$ for $n >2$, then $Comm(G)$ fails to be path-connected.

Namely, it follows from the work below that translates of $ Assoc(T)$ in $Assoc(G)$ under the natural action of the Weyl group give a ``good'' approximation for the space of all commuting $n$-tuples, $Comm(G)$. This process is identifying commuting $n$-tuples in terms of subspaces of $Assoc(G)$. Furthermore, the space $Comm(G)$ can be regarded as the space of all ordered finite sets of generators for all finitely generated abelian subgroups of $G$ modulo the single relation that the identity element $1_G$ of $G$ can be omitted in any set of generators.

The next step is to see how $Comm(G)$ relates to earlier classical constructions. In particular, $G \times_{NT}Assoc(T)$ maps  naturally to $Comm(G)$ and is an infinite dimensional analogue
of the natural map $$\theta: G \times_{NT}T \to G$$ given by $$\theta(g,t) = gtg^{-1}= t^g,$$
with more details given in Section \ref{section:x2g}.

\begin{defn}\label{definition: Borel to Comm(G)}
Define $$\Theta: G \times_{NT}Assoc(T) \to Comm(G)$$ by the formula
$$\Theta(g, (t_1,\dots, t_n)) = (t_1^g,\dots, t_n^g)$$ for all $n$.

\end{defn}

The space $Comm(G)$ is ``roughly" the union of orbits of $Assoc(T)$ in $Assoc(G)$ as will be seen below. The next step is a direct observation with proof omitted.
\begin{prop}\label{prop: continuous map to Comm(G)}
Let $G$ be a compact connected Lie group with maximal torus $T$ and Weyl group $W$.  The map
$$\Theta: G \times_{NT}Assoc(T) \to Comm(G)$$
is well-defined and continuous. Furthermore, there is a commutative diagram

\[
\begin{CD}
G \times_{NT}T  @>{\theta}>> G     \\
 @V{1 \times E}VV    @VV{E}V       \\
G \times_{NT}Assoc(T)  @>{\Theta}>> Comm(G)\\
 @V{1}VV    @VV{\subseteq}V       \\
G \times_{NT}Assoc(T)  @>{\Theta}>> Assoc(G).\\
\end{CD}
\] 

In addition, the map $$\Theta: G \times_{NT} Assoc(T) \to Comm(G)$$
factors through one path-component as
$\Theta: G \times_{NT} Assoc(T) \to Comm(G)_{1_G}$, and
there is a commutative diagram for all $n \geq 0$ given as follows:

\[
\begin{CD}
G \times T^m @>{\theta_m}  >> \Ho(\Z^m,G)_{1_G} \\
 @VV{}V    @VV{}V       \\
G \times_{NT}Assoc(T)  @>{\Theta}>> Comm(G)_{1_G}.\\
\end{CD}
\] 

\end{prop}

A second useful naturality result concerning $Comm(G)$ is satisfied for certain morphisms of Lie groups which preserve a choice of maximal torus.  The next proposition is again an inspection of the definitions.
\begin{prop}\label{prop: naturality}

Let $T_H$ and $T_G$ be maximal tori of compact Lie groups $H$ and $G$, respectively. If a continuous homomorphism $$f: H \to G$$  is a morphism of Lie groups such that $f(T_H) \subset T_G$, then there is an induced commutative diagram
\[
\begin{CD}
H \times_{N_{T_H}}Assoc(T_H)  @>{\Theta}>> Comm(H)\\
 @V{f \times Assoc(f)}VV    @VV{Comm(f)}V       \\
G \times_{N_{T_G}}Assoc(T_G)  @>{\Theta}>> Comm(G).\\
\end{CD}
\] 

\end{prop}

The next feature is that the spaces $G \times_{NT}Assoc(T)$, and $Comm(G)$ admit stable decompositions which are 
compatible as well as provide simple, direct computations.

\begin{defn}\label{definition: singular subspaces of commuting n-tuples}
Given a Lie group $G$, define  
$$S_k(G) =S(\Ho(\Z^k,G)) \subset \Ho(\Z^k,G)$$ 
as those $k$-tuples for which at least one coordinate is equal to the identity element $1_G \in G$. Define $$\widehat{\Ho}(\Z^k,G) = \Ho(\Z^k,G)/S_k(G).$$
\end{defn} 

\begin{rmk}
Note that some of the theorems in this paper require that the inclusions $S_k(G)\hookrightarrow \Ho(\Z^k,G)$ are cofibrations. For this to be true it suffices to assume that $G$ is a closed subgroup of $GL_n(\C)$ (see \cite[Theorem 1.5]{adem2007commuting}), and we assume this implicitly whenever necessary.
\end{rmk}

In what follows, $\Sigma(X)$ denotes the suspension of the pointed space $X$.
The following theorem was proven in \cite{adem2007commuting}:
\begin{thm}\label{theorem: stable splitting of Hom(Zn,G)}
Let $G$ be a  Lie group. Then there are homotopy equivalences

$$\Sigma(\widehat{\Ho}(\Z^n,G)) \vee \Sigma(S_n(G)) \to \Sigma(\Ho(\Z^n,G)),$$
and
$$\bigvee_{1 \leq j \leq n}\bigvee_{\binom {n}{j}}\Sigma(\widehat{\Ho}(\Z^j,G)) \to \Sigma(\Ho(\Z^n,G)).$$

\end{thm}

The first result concerning the suspension of $Comm(G)$ is as follows.
\begin{thm}\label{thm:stable decomp f x2g}
Let $G$ be a  Lie group. Then there is a homotopy equivalence
$$\Sigma(Comm(G)) \to \bigvee_{n \geq 1}\Sigma(\widehat{\Ho}(\Z^n,G)).$$ 
\end{thm}

The spaces $\widehat{\Ho}(\Z^q,G)$ for $ 1 \leq q \leq n$ give all of the stable summands of $\Ho(\Z^n,G)$, and so these spaces determine the homotopy type of the suspension of $\Ho(\Z^n,G)$. Thus the space $Comm(G)$ captures all of the stable structure of the spaces $\Ho(\Z^n,G)$ in a ``minimal'' way by splitting into a wedge after suspending of exactly one summand 
 $\widehat{\Ho}(\Z^q,G)$  for each integer $q$. Namely, $Comm(G)$ can be regarded  as ``the smallest space'' which is stably equivalent to a wedge of all of the ``fundamental summands'' $\widehat{\Ho}(\Z^q,G)$.

On the other hand, it is natural to ask how the``fundamental summands" are connected with the previous construction
$G\times_{NT} Assoc(T)$. Here recall the definition of the $q$-fold smash product $$\widehat{T}^q=
T^q/S(\Ho(\Z^q,T)).$$

\begin{thm}\label{thm:first approximation INTRO}
Let $G$ be a  Lie group. Then there are homotopy equivalences
$$\Sigma(G\times_{NT} Assoc(T)) \to \Sigma \bigg( G/T \vee \Big(\bigvee_{q \geq 1} 
G\times_{NT} \widehat{T}^q/(G/NT)\Big) \bigg).$$
\end{thm}

There is a natural extension of $Comm(G)$ to all finitely generated nilpotent subgroups of nilpotence class at most $q$. There is a space that assembles all the spaces $\Ho(F_n/\Gamma^q,G)$ into a single space, denoted $X(q,G)$ (with formal definition stated in Definition \ref{definition: XqG}), where $\Gamma^q$ denotes the $q$-th stage of the descending central series of $F_j$ for all $j$.

Notice that the spaces $X(q,G)$ give a filtration of $J(G)$ with 
$$Comm(G)  = X(2,G) \subset X(3,G) \subset \cdots \subset X(q,G) \subset \cdots \subset J(G).$$ 
The spaces $X(q,G)$ in the filtration also admit stable decompositions as in the case of $X(2,G)$; see Section \ref{section: f x2g}. 

\begin{thm} \label{thm: stable decompositions of X(q G)}
Let $G$ be a compact, connected Lie group.
Then there are homotopy equivalences
$$\Sigma(\Ho(F_n/\Gamma^q,G)) \to \bigvee_{1\leq q \leq n}(\bigvee_{\binom{n}{j}}\Sigma(\widehat{\Ho}(F_j/\Gamma^q,G)),$$ and

$$\Sigma(X(q,G)) \to \bigvee_{1\leq j < \infty}\Sigma(\widehat{\Ho}(F_j/\Gamma^q,G)).$$

\end{thm}

\begin{rmk}\label{rmk: stable splittings }
The above raises the question of identifying the stable wedge summands of other spaces of homomorphisms, or spaces of representations. For example, these methods also inform on the spaces of homomorphisms from a free group on $n$ letters modulo the stages of either the descending central series or mod-$p$ descending central series. It is currently unclear
whether these spaces inform on representation varieties associated to fundamental groups of Riemann surfaces, but it seems likely that these methods will inform on representation varieties for braid groups. Furthermore,  the space $Comm(G)$ maps naturally by evaluation onto the space of closed, finitely generated abelian subgroups of $G$ topologized by the Chabauty topology as in \cite{bridson.h.k}. It is natural to ask whether this map is a fibration with fiber $J(T)$. 
\end{rmk}

Theorem \ref{thm:first approximation INTRO} is used to obtain the next result.
\begin{thm} \label{thm: further.stable decompositions}
Let $G$ be a compact, simply-connected Lie group, and assume that all spaces have been localized such that the order of the Weyl group has been inverted (localized away from $|W|$). Then there are homotopy equivalences
$$ \Sigma (G\times_{NT} \widehat{T}^q/(G/NT)) \to \Sigma(\widehat{\Ho}(\Z^q,G))$$ for all $q \geq 1$ as long as the order of the Weyl group has been inverted.
\end{thm}

These results can be used to directly work out the homology of
$\widehat{\Ho}(\Z^n,G)$ for all $n$, and thus ${\Ho}(\Z^n,G)$ as long as the order of the Weyl group has been inverted.

The following theorem reduces a computation of the homology of $Comm(G)$ to standard methods, and is proven below.

\begin{thm}\label{theorem: the surjection map to x(2,g)}
Let $G$ be a  simply-connected compact Lie group. Then the map
$$\Theta: G \times_{NT}Assoc(T) \to Comm(G)$$ has homotopy theoretic
fibre which is simply-connected, and has reduced homology which is
entirely torsion with orders dividing the order of the Weyl group $W$ of $G$.

\end{thm}

Finer stable decompositions of the space $Comm(G)$ and $\Ho(\Z^n,G)$, as well as further analysis of the combinatorics arising here in homology will appear in a sequel to this paper.

Homological applications of the above decompositions are given next.
Let $\Gamma$ be a discrete group and $R[ \Gamma]$ be the group ring of $\Gamma$ over the commutative ring $R$. Let $M$ be a left $R[\Gamma]$-module. Recall that the module of coinvariants $M_{\Gamma}$ is the largest quotient of $M$ such that $\Gamma$ acts trivially on that quotient. That is, $M_{\Gamma} = M/I$, where $I$ is the group generated by $\gamma \cdot m-m$, for all $\gamma \in \Gamma$ and $m \in M$. Furthermore, 
$$\T [V]$$ 
denotes the tensor algebra generated by the free $R$-module $V$. Throughout this section  $R = \Z[1/|W|]$.

\begin{thm}\label{thm: homology of x2g}
Let $G$ be a compact, connected Lie group with maximal torus $T$ and Weyl group $W$. Then there is an isomorphism in homology $$H_{\ast}( Comm(G)_{1_G}; R) \to H_{\ast}(G/T;R) \otimes_{R[W]}  \T[V]$$
where $V$ is the reduced homology of the maximal torus $T$.
\end{thm}

Let $H^U_{\ast}$ and $\T_U$ denote ungraded homology and the ungraded tensor algebra, respectively. Then the following theorem holds.
\begin{thm}\label{theorem:ungraded homology of X2G INTRO}
Let $G$ be a compact and connected Lie group with maximal torus $T$ and Weyl group $W$. Then there is an isomorphism in ungraded homology
\[H_{\ast}^U( Comm(G)_1; R) \to  \T_U [V] \]  where $V$ is the reduced homology of the maximal torus $T$.
\end{thm}
If $G$ is a compact and connected Lie group such that every abelian subgroup of $G$ is in a maximal torus, then the following corollary holds.
\begin{cor}\label{cor: homology of x2g}
Let $G$ be a compact and connected Lie group with maximal torus $T$ and Weyl group $W$ such that every abelian subgroup of $G$ is in a maximal torus. Then there is an isomorphism in homology
\[H_{\ast}( Comm(G); R) \to  H_{\ast}(G/T;R) \otimes_{R[W]}  \T[V] \]
where $V$ is the reduced homology of the maximal torus $T$.
\end{cor}

Examples of this last corollary are given below for the cases of $U(n)$, $SU(n)$, $Sp(n)$, $Spin(n)$, and the $5$ exceptional compact simply-connected simple Lie groups $G_2, F_4, E_6, E_7,$ and $E_8$ in Sections
\ref{section:Un.SUn} and \ref{section:exceptional Lie groups}. The Poincar\'e series for the ungraded homology groups are also recorded in Sections \ref{section:Un.SUn} and \ref{section:exceptional Lie groups}.

Let $G$ be a compact, connected Lie group with maximal torus $T$ and Weyl group $W$.   Let $C$ denote the homology of $G/T$ with real coefficients where $C_i = H_i(G/T;\R)$,
and $C^*$ is defined by $C^i = H^i(G/T;\R)$.

The Hilbert-Poincar\'e series for the rational cohomology of $Comm(G)$ is computed as follows.

Consider the $W$-action on the exterior algebra $E=\wedge \mathbb R^n= H_*(T; \R)$ as well as the reduced exterior algebra $\barE$, namely, the reduced homology of $T$, given by $\bigoplus_{k=1}^n \wedge^k \RR^n$. Then $W$ acts diagonally 
on the $\mathbb R$-dual $\TTT^*[\barE]$ of the tensor  tensor algebra generated by $\barE$,
the cohomology of $J(T)$. The group $W$ acts diagonally on the tensor product $C^* \otimes \TTT^*[\barE]$ giving the action on the cohomology of $G/T \times J(T)$.

These modules are naturally trigraded by both homological degree as well as tensor degree arising from the tensor algebra, and have a natural $\mathbb N^3$-trigrading as follows.
First, the homology of $J(T)$ is given by the direct sum 
$$\R \oplus \sum_{\substack{j = k_1+\cdots+k_m\\   \ k_q > 0\\}} ( \wedge^{k_1} \RR^n \otimes \cdots \otimes \wedge^{k_m} \RR^n).$$   Then 
 the sub-module of elements of tri-degree $(i,j,m)$ in the homology 
 of $G/T \times J(T)$ for fixed $m > 0$  is given by the direct sum 
$$\sum_{\substack{j = k_1+\cdots+k_m\\   \ k_q > 0\\}} (C_i \otimes \wedge^{k_1} \RR^n \otimes \cdots \otimes \wedge^{k_m} \RR^n),$$  
such that $1 \leq q \leq m $, $\wedge^k \mathbb R^n$ lies in homological degree $k$.
To compute in cohomology, the $\mathbb R$-dual of $\wedge^{k_1} \RR^n \otimes \cdots \otimes \wedge^{k_m} \RR^n$ lies in homological degree $j = k_1+\cdots+k_m$ of $\TTT^*[\barE]$ as well as tensor degree  $m>0$. The special case with $m = 0$ is by convention $C_i \otimes \mathbb R \cong C_i $ in tri-degree $(i,0,0)$.

\begin{defn} \label{defn: trigraded Hilbert-Poincare series}
Consider the module 
$$ M_{(i,j,m)} = \sum_{\substack{j = k_1+\cdots+k_m\\   \ k_q > 0 \\}} (C_i \otimes \wedge^{k_1} \RR^n \otimes \cdots \otimes \wedge^{k_m} \RR^n) \ \mbox{for} \ m >0$$ 
together with $$M_{(i,0,0)} = C_i \otimes \mathbb R.$$

The diagonal action of $W$ 
on the tensor product $C^* \otimes \TTT^*[\barE]$ over $\RR$
respects this  {\it $\NN^3$-trigrading}.
The tri-graded Hilbert-Poincar\'e series for the module of invariants is defined by
$$
\Hilb\left( \left(C^* \otimes \TTT^*[\barE] \right)^W, q, s, t \right)
= \sum_{\substack{ i,m\geq 0 \\ j = k_1+\cdots+k_m  \\}} \dim_{\R} (M_{(i,j,m)}^W)q^is^jt^m.
$$

\end{defn}

Let $d_1,\dots,d_n$ be \textit{the fundamental degrees of $W$} (which are defined either in the Appendix or \cite[\S 4.1]{broue.reflextion.gps}).
Note that the degrees for algebra generators of polynomial rings such as the cohomology ring of $BT$ and their images in $H^*(G/T;\mathbb R)$ are given by doubling ``the fundamental degrees of $W$'' $(d_1,...., d_n)$ in the cited work by Shephard and Todd et al. The ``fundamental degrees of $W$'' are doubled here in order to correspond to the usual conventions for ``topological gradings'' associated to the cohomology of the topological space $BT$.

The following result is proven in the Appendix.

\begin{thm}\label{thm:souped-up-Molien INTRO}
If $G$ is a compact, connected Lie group with maximal torus $T$, and Weyl group $W$, then
$$
\begin{aligned}
&\Hilb\left( \left(C^* \otimes \TTT^*[\barE] \right)^W , q, s, t \right) \\
&\quad =\frac{ \prod_{i=1}^n (1-q^{2d_i}) }{ |W| }\sum_{w \in W} \frac{1}{\det(1-q^2w) \left( 1-t(\det(1+sw)-1) \right)}.
\end{aligned}
$$
\end{thm}

\begin{ex}\label{ex: example Hilbert-Poincare series for U(2)}
The Hilbert-Poincar\'e series for $Comm(G)$ where $G=U(2)$ is worked out in two ways.
One way is the formula given in Theorem \ref{thm:souped-up-Molien INTRO}. The second way is by enumerating the representations of $W$ which occur in
$C^* \otimes \TTT^*[\barE] $.

\begin{enumerate}
\item 
The Weyl group $W$ is $\Sigma_2$ with elements $1$, and $w\neq 1$.
\item The homology of the space $G/T = U(2)/T = S^2$ is $\mathbb R$ in degrees zero and two, and is $\{0\}$ otherwise.
\item The degrees $(d_1,d_2)$ in Theorem \ref{thm:souped-up-Molien INTRO} are given by $(d_1,d_2) = (1,2)$.  
\item The sum in Theorem \ref{thm:souped-up-Molien INTRO} runs over $w$ and $1\in W$.
\end{enumerate} 
Then the formula given in Theorem \ref{thm:souped-up-Molien INTRO}
$$
\Hilb\left( \left(C^* \otimes \TTT^*[\barE] \right)^W , q, s, t \right) =  \frac{ (1-q^2)(1-q^4)}{2}(A_1 + A_w), $$ 
where  $$\ds A_1=  \frac{1}{(1-q^2)^2(1-t[(1+s)^2-1])}$$ and $$\ds A_w =  \frac{1}{(1-q^2)(1+q^2)(1-t[(1+s)(1-s)-1])}. $$
Thus 
$$\frac{ (1-q^2)(1-q^4)}{2}(A_1) = \frac{1+q^2}{2(1-t(s^2+2s))} \mbox{ and }  \ \frac{ (1-q^2)(1-q^4)}{2}(A_w) = \frac{1-q^2}{2(1+s^2t)}.$$  
The Hilbert-Poincar\'e series is then given by
\begin{align*}
& \Hilb\left( \left(C^* \otimes \TTT^*[\barE] \right)^W , q, s, t \right) = 
\frac{1+q^2}{2(1-t(s^2+2s))} + \frac{1-q^2}{2(1+s^2t)}. \\
\end{align*}

From this information, it follows that the coefficient of $t^m, \  m > 0$, is

$$
\begin{aligned}
&\frac{1}{2}\left[  (1+q^2)(s^2+2s)^m+ (1-q^2)(-s^2)^m \right ] \\
&\quad = \sum_{1 \leq j \leq m}2^{j -1}\binom{m}{j}s^{2m-j} \ + \left\lbrace\begin{array}{lr}
s^{2m} & \mbox{if $m$ is even, and } \\
q^2s^{2m} &  \mbox{if $m$ is odd. }\\
 \end{array}\right..
\end{aligned}
$$

Keeping track of the representations of $W$, an exercise left to the reader, gives an independent verification of the formula in Theorem \ref{thm:souped-up-Molien INTRO} for this special case.
\end{ex}

Theorem \ref{thm: further.stable decompositions}
states that if $G$ is a compact, connected Lie group, then there are homotopy equivalences
$$ \Sigma (G\times_{NT} \widehat{T}^m/(G/NT)) \to \Sigma(\widehat{\Ho}(\Z^m,G))$$ for all $m \geq 1$
as long as the order of the Weyl group has been inverted (spaces are localized away from $|W|$).  The reduced, real cohomology of
$G\times_{NT} \widehat{T}^m/(G/T)$ as well as
$\widehat{\Ho}(\Z^m,G))$ is given by $$\sum_{\substack{j = k_1+\cdots+k_m\\   \ i \geq 0  \\}} (C_i \otimes \wedge^{k_1} \RR^n \otimes \cdots \otimes \wedge^{k_m} \RR^n)^W= 
\sum_{\substack{j = k_1+\cdots+k_m\\   \ i \geq 0  \\}} M_{(i,j,m)})^W.$$  By Theorem \ref{theorem: stable splitting of Hom(Zn,G)}, there are homotopy equivalences $$\bigvee_{1 \leq j \leq m}\bigvee_{\binom {m}{j}}\Sigma(\widehat{\Ho}(\Z^j,G)) \to \Sigma(\Ho(\Z^m,G)).$$

\begin{cor}\label{cor: Hilbert-Poincare series for stable splitting of Hom(Zn,G)}
Let $G$ be a compact, connected  Lie group as in Definition
\ref{defn: trigraded Hilbert-Poincare series}. Then there are 
additive isomorphisms

$$ \widetilde{ H}^d(\Ho(\Z^m,G);\mathbb R) \to \sum_{1 \leq s \leq m}
\sum_{i+j = d}\bigg(  \sum_{\substack{j = k_1+\cdots+k_s\\   \ i \geq 0  \\}}  \oplus_{\binom {m}{s}} (M_{(i,j,s)})^W\bigg). $$
\end{cor}

\begin{rmk}\label{rmk: Poincare series for Hom(Z^n,G)}
The analogue of the Hilbert-Poincar\'e series for 
$\widetilde{ H}^*\Ho(\Z^m,G)$ can be given in terms of
that for $Comm(G)$.
\end{rmk}

\textbf{Acknowledgment:} The authors thank V. Reiner for including the appendix in this paper along with his formula giving the Hilbert-Poincar\'e series for $Comm(G)$.

\section{The James construction}\label{section: James construction}
 
{ In this section we give a quick exposition of some properties of the James construction $J(X)$, the free monoid generated by the space $X$ with the base-point in $X$ acting as the identity as given in the next definition. The properties listed here are well-known, and are stated for the convenience of the reader.

\begin{defn}{\label{jamesdefn}}
The \textit{James construction} or \textit{James reduced product}on $X$ is the space
\[J(X):= \bigsqcup_{n\geq 0} X^n /\sim ,\]
where $\sim$ is the equivalence relation generated by 
$(x_1,\dots,x_n) \sim (x_1,\dots,\widehat{x_i},\dots,x_n)$ if $x_i=\ast,$
with the convention that $X^0$ is the base-point $\ast$.
\end{defn}
 }

One of the most important properties of $J(X)$ is the existence of a map
$$\theta: J(X) \to \Omega \Sigma(X)$$ which is a homotopy equivalence in case $X$  has the homotopy type of a path-connected CW-complex \cite{james1955reduced}. The singular homology of $J(X)$ is described next.

Recall the homology of $J(X)$ for any path-connected space $X$ of the homotopy type of a CW-complex as follows. Let $R$ be a commutative ring with $1$. Assume that homology is taken with coefficients in the ring $R$ where the homology of $X$ is assumed to be $R$-free. Finally, let 
$$W = \widetilde{H}_*(X;R)$$
the reduced homology of $X$, and 
$$\T[W]$$ 
be the tensor algebra generated by $W$. 
Then the Bott-Samelson theorem gives an isomorphism of algebras
$$\T[W] \to H_*(J(X);R).$$  Although their results give a more precise description of this isomorphism as Hopf algebras,
that additional information is not used here.

Furthermore, if $X$ is of finite type, the Hilbert-Poincar\'e series for $H_*(X;R)$ is defined in the usual way as

$$\Hilb(X,t) =  \sum_{0 \leq n} d_nt^n,$$ 
where $d_n$ is the rank of $H_n(J(X);R)$ as an $R$-module (with the freeness assumptions made above). 
Then the following formula holds:
$$\Hilb(J(X),t) = 1/(1-(\Hilb(X,t)-1))= 1/(2-\Hilb(X,t)).$$

An example is described next where $X$ is a finite product of circles.
\begin{ex}\label{ex:ungraded homology of X2G}
If $$X = (S^1)^m, $$ then $$\Hilb(X,t) = (1+t)^m= \sum_{0 \leq j \leq m} \binom{{m}}j t^j,$$ thus
$$ \Hilb(J(X),t) = 1/(2- (1+t)^m) = \frac{1}{1- \sum_{1 \leq j \leq m} \binom{{m}}j t^j}.$$
\end{ex}

This procedure will be used to describe the ungraded homology groups of the spaces $Comm(G)_{1}$ for various choices of the group $G$.

\section{Proof of Theorem \ref{thm:first approximation INTRO}}\label{section:stable.decomp.Borel.constr}

Begin by proving the following theorem where it will be tacitly assumed that all spaces here are of the homotopy type of a connected CW-complex.
 
\begin{thm}{\label{decomp.thm}}
Let $Y$ be a $G$-space such that the projection $Y \longrightarrow Y/G$ is a locally trivial fibre bundle, and $X$ a $G$-space with fixed base-point $\ast$. Then there is a homotopy equivalence
\[\Sigma (Y\times_G J(X)) \simeq \Sigma \big(Y/G \vee (\bigvee_{n \geq 1} (Y\times_G \widehat{X}^n) /(Y\times_G \ast ))\big).\]
\end{thm}

Theorem \ref{thm:first approximation INTRO} will be an immediate corollary of Theorem \ref{decomp.thm}. The proof of the theorem is given below by a list of lemmas, see also \cite{may.generalized.splitting.thm}. The current proof is an equivariant splitting.

Assume $X$ and $Y$ are $G$-CW complexes satisfying the conditions of Theorem \ref{decomp.thm}. Consider the map 
$$H:J(X) \to J(\bigvee_{1 \leq q< \infty}\widehat{X}^q)$$ 
given by
\[(x_1,\dots,x_n) \mapsto \prod_{I \subset [n]}x_I,\]
where $x_I=x_{i_1}\wedge \cdots \wedge x_{i_q}$ for $I=(i_1,\dots,i_q)$ running over all admissible sequences in $[n]$, i.e. all sequences of the form $(i_1<\cdots<i_q)$, with $x_I$ having the left lexicographic order. 

Next filter by defining 
$$F_N J(\bigvee_{1 \leq q< \infty}\widehat{X}^q)$$ 
to be the image of 
$$J(\bigvee_{1 \leq q \leq N}\widehat{X}^q).$$
There is also a filtration of $J(X)$ given by
$$ {\ast}=J_0(X)\subset J_1(X)\subset J_2(X) \subset \cdots \subset J_q(X)\subset \cdots \subset J(X),$$
where $J_q(X)$ is the image of $\bigsqcup_{1 \leq i \leq q} X^i/\sim $ in $J(X)$.

The next lemmas follow by inspection.

\begin{lemma}[]{\label{filtration}}
If $X$ has a base-point, then the map $$H:J(X) \to J(\bigvee_{1 \leq q< \infty}\widehat{X}^q)$$
restricts to a map $$H:J_N(X) \to F_NJ(\bigvee_{1 \leq q< \infty}\widehat{X}^q)= J(\bigvee_{1 \leq q \leq N}\widehat{X}^q),$$ and preserves filtrations. Furthermore, the map $H:J_N(X) \to J(\bigvee_{1 \leq q \leq N}\widehat{X}^q)$ is $G$-equivariant.
\end{lemma}

\begin{lemma}[]{\label{filtration.and.maps}}
Let $X$ be a pointed space where the group $G$ acts on the space $Y$ as well as on $X$ fixing the base-point of $X$. Then there is a map $$\widehat{H}: Y \times J(X)  \to J(\bigvee_{1 \leq q< \infty}(Y \times \widehat{X}^q))$$ defined by 
$$\widehat{H}(y;(x_1, \cdots, x_q)) = \prod _{1 \leq i_1< \cdots< i_t \leq q}(i,(x_{i_1},x_{i_2}, \cdots, x_{i_t})).$$ 
Furthermore, there is an induced map
\[
\begin{CD}
Y \times_GJ_N(X)/ Y\times_{G}\ast @>{}>> J(\bigvee_{1 \leq q \leq N}(Y \times_G \widehat{X}^q/ Y\times_{G}\ast)) \\
\end{CD}
\] 
together with a strictly commutative diagram

\[
\begin{CD}
Y \times_GJ_{N-1}(X)/ Y\times_{G}\ast @>{}>> J(\bigvee_{1 \leq q \leq N-1}(Y \times_G \widehat{X}^q/ Y\times_{G}\ast))\\
 @VV{}V    @VV{J(\mbox{inclusion})}V       \\
 Y \times_GJ_N(X)/ Y\times_{G}\ast @>{}>> J(\bigvee_{1 \leq q \leq N}(Y \times_G \widehat{X}^q/ Y\times_{G}\ast))\\
 @VV{}V    @VV{J(\mbox{projection})}V       \\
 Y \times_G (\widehat{X}^N)/ Y\times_{G}\ast @>{1}>> (Y \times_G \widehat{X}^N/ Y\times_{G}\ast).\\
\end{CD}
\]


\end{lemma}


Observe that passage to adjoints gives a commutative diagram as follows:

\[
\begin{CD}
\Sigma(Y \times_GJ_{N-1}(X)/ Y\times_{G}\ast )  @>{}>> \Sigma(\bigvee_{1 \leq q \leq N-1}(Y \times_G \widehat{X}^q/ Y\times_{G}\ast))\\
 @VV{}V    @VV{\Sigma(\mbox{inclusion})}V       \\
\Sigma (Y \times_GJ_N(X)/ Y\times_{G}\ast )@>{}>> \Sigma(\bigvee_{1 \leq q \leq N}(Y \times_G \widehat{X}^q/ Y\times_{G}\ast))\\
 @VV{}V    @VV{\Sigma(\mbox{projection})}V       \\
 \Sigma (Y \times_G (\widehat{X}^N)/ Y\times_{G}\ast) @>{1}>>\Sigma (Y \times_G \widehat{X}^N/ Y\times_{G}\ast).\\
\end{CD}
\]

The top horizontal arrow is an equivalence for $N-1= 0$ by hypothesis, and in general by the evident inductive hypothesis. The bottom arrow is an equivalence by inspection. Furthermore, the columns give a morphism of cofibration sequences by hypothesis. Thus, the middle arrow is an equivalence assuming that all spaces are of the homotopy type of a CW-complex.


Let $G$ be a compact and connected Lie group with maximal torus $T$. Consider the quotient space $G\times_{NT}J(T)$, where $NT$ acts diagonally on the product $G\times J(T)$. Then the following are immediate corollaries of Theorem \ref{decomp.thm}. Note that Theorem \ref{thm:first approximation INTRO} is the following corollary.
\begin{cor}{\label{app1}}
Let $G$ be a compact and connected Lie group with maximal torus $T$, and $NT$ acting on $T$ by conjugation and on $G$ by group multiplication. There is a homotopy equivalence
\[\Sigma (G\times_{NT} J(T)) \simeq \Sigma (G/NT \vee (\bigvee_{n \geq 1} G\times_{NT} \widehat{T}^n /G\times_{NT} \{1\} )).\]
\end{cor}

\begin{cor}{\label{app1/G}}
With the same assumptions of Corollary \ref{app1}, there is a homotopy equivalence
\[\Sigma (G\times_{NT} J(T)/(G\times \{1\})) \simeq \Sigma (G/NT \vee (\bigvee_{n \geq 1} G\times_{NT} \widehat{T}^n /G\times_{NT} \{1\} ))/(G\times \{1\}).\]
\end{cor}

\section{Proof of Theorem \ref{thm: stable decompositions of X(q G)}}\label{section: f x2g}

The space $Comm(G)$ is a special case of a more general construction on $G$. Let $F_n$ be a free group on $n$ letters. Consider the descending central series of $F_n$ denoted 
$$ \cdots \subseteq \Gamma^q(F_n) \subseteq \cdots \subseteq \Gamma^3(F_n) \subseteq \Gamma^2(F_n) \subseteq F_n.$$
Then $\Z^n=F_n/\Gamma^2(F_n)=F_n/[F_n,F_n]$. Consider the quotients $F_n/\Gamma^q(F_n)$. There are spaces of homomorphisms $\Ho(F_n/\Gamma^q(F_n),G)$ which can be realized as subspaces of $G^n$ by identifying $f \in \Ho(F_n/\Gamma^q(F_n),G)$ with $(g_1,\dots,g_n)\in G^n$, where $f(x_i)=g_i$ and $x_1,\dots,x_n$ are the generators of $F_n/\Gamma^q(F_n)$. The following definition extends the definition of $Comm(G)$ to include spaces of homomorphisms $
\Ho (F_n/\Gamma^q(F_n),G)$.

\begin{defn}\label{definition: XqG}
Define spaces $X(q,G) \subset Assoc(G)$ for $q\geq 2$ as follows
\[X(q,G):=\bigsqcup_{n\geq 1} \Ho (F_n/\Gamma^q(F_n),G)/{\sim}\]
where $\sim$ is the equivalence relation defined by
\[(x_1,\dots,x_n) \sim (x_1,\dots,x_{i-1},\widehat{x_i} ,x_{i+1},\dots,x_n)\text{ if }  x_i=\ast, \]
as in Definition \ref{definition: spaces of all commuting n-tuples}. Note the space $X(q,G)$ coincides with $Comm(q,G)$ in Definition \ref{definition: spaces of all commuting n-tuples}.
\end{defn}

The inclusions $\Ho(F_n/\Gamma^q,G) \hookrightarrow \Ho(F_n/\Gamma^{q+1},G)$ obtained by precomposition of maps, induce inclusions $X(q,G) \subset X(q+1,G)$ for all $q \geq 2$. Hence, there is a filtration of $Assoc(G)$ by the spaces $X(q,G)$
$$Comm(G) = X(2,G) \subset X(3,G) \subset \cdots \subset X(q,G) \subset \cdots \subset J(G)=Assoc(G). $$

The space $X(q,G)$ stably splits if suspended once.
\begin{thm}\label{stable.x2g.prop}
Let $G$ be a compact and connected Lie group. There are homotopy equivalences

\[\Sigma Comm(G) \simeq \Sigma \bigvee_{n\geq 1} \widehat{\Ho}(\mathbb{Z}^n,G),\]

and the spaces $X(q,G)$

\[\Sigma X(q,G) \simeq \Sigma \bigvee_{n\geq 1} \widehat{\Ho}(F_n/\Gamma^q(F_n),G).\]

\end{thm}

\begin{proof}

Let $ q \geq 2$. Define a filtration of $X(q,G)$ as follows
\[F_1X(q,G) \subseteq F_2X(q,G) \subseteq \cdots  \subseteq F_nX(q,G) \subseteq F_{n+1}X(q,G) \subseteq  \cdots,\] 
where by definition $F_nX(q,G)$ is the image of
$$\bigcup_{1 \leq t \leq n} \Ho(F_t/\Gamma^q(F_t),G)$$ 
in $X(q,G).$
Notice that in the special case of $q = 2$,
$\Z^n = F_n/\Gamma^2(F_n)$.

Hence, $F_nX(q,G)$ is the space of words of length at most $n$, such that the letters of the words 
are in $\Ho(F_n/\Gamma^q(F_n), G)$ (or $\Z^n$ in case $q = 2$).

In case $ q = 2$, it follows that the quotient of two consecutive stages of the filtration is 
$$F_{q+1}X(2,G)/F_{q}X(2,G) \simeq \Ho(\mathbb{Z}^{q+1},G)/S(\Ho(\mathbb{Z}^{q+1},G))=: \widehat{\Ho}(\mathbb{Z}^{q+1},G).$$ In case $ q > 2$, it was verified in \cite{adem2007commuting} that the inclusions
$$\Ho(F_{n-1}/\Gamma^q(F_{n-1}),G) \to \Ho(F_n/\Gamma^q(F_n),G)$$ are closed cofibrations. By the above remarks,
these inclusions are split. Thus there are cofibration sequences
$$\Ho(F_{n-1}/\Gamma^q(F_{n-1}),G) \hookrightarrow \Ho(F_n/\Gamma^q(F_n),G) \to C,$$ 
where $C$ is the cofiber $\Ho(F_{n}/\Gamma^q(F_{n}),G)/\Ho(F_{n-1}/\Gamma^q(F_{n-1}),G)$, and the sequences are split via the natural projection map 
$$p_n: \Ho(F_{n}/\Gamma^q(F_{n}),G) \to \Ho(F_{n-1}/\Gamma^q(F_{n-1}),G)$$
which deletes the $n$-th coordinate.

Notice that $Assoc(G) = J(G)$ splits after one suspension as 
$$\Sigma \bigvee_{1 \leq n < \infty} \widehat{G}^{n}.$$ 
The splitting of $X(q,G)$ is induced by inspection. It follows that $X(q,G)$ splits as stated.
\end{proof}

If $G$ is a closed subgroup of $GL_n(\C)$, then A. Adem and F. Cohen \cite{adem2007commuting} show that there is a homotopy equivalence
\[\Sigma (\Ho (\mathbb{Z}^n,G)) \simeq \bigvee_{1\leq k \leq n}\Sigma\big(\bigvee^{n \choose k} \Ho (\mathbb{Z}^k,G)/S(\Ho(\mathbb{Z}^{k},G)) \big).\]
Therefore, for these Lie groups, it is possible to obtain all the stable summands of $\Ho(\mathbb{Z}^n,G)$ from the stable summands of $ Comm(G)$.

\section{Proof of Theorems \ref{thm: further.stable decompositions} \& \ref{theorem: the surjection map to x(2,g)}}\label{section:x2g}

Let $G$ be a compact and connected Lie group. Let $\Ho(\mathbb{Z}^n,G)_{1_G}$ denote the connected component of the trivial representation $\mathbf{1}=(1,\dots,1)$ in $\Ho(\mathbb{Z}^n,G)$. Since $T^n$ consists of commuting $n$-tuples, is path-connected, and contains $\mathbf{1}$, it is a subspace of $\Ho(\mathbb{Z}^n,G)_1 \subseteq G^n$. In addition, $G$ acts on the space $T^n$ by coordinatewise conjugation 
\begin{align*}
\theta_n: G\times T^n &\to \Ho(\mathbb{Z}^n,G)_{1_G} \\
g\times (t_1,\dots,t_n)& \mapsto (t_1^g,\dots,t_n^g),
\end{align*}
where $t^g=gtg^{-1}$. An $n$-tuple $(h_1,\dots,h_n)$  of elements of $G$ is in $\Ho(\mathbb{Z}^n,G)_{1_G}$ if and only if there is a maximal torus such that all $h_i$ are in that torus.  
Since all maximal tori in $G$ are conjugate (see \cite{adams.lectures.on.lie.groups}), it follows that the map $\theta_n$ is surjective.

The maximal torus $T$ acts diagonally on the product $G \times T^n$ 
\[t\cdot (g,t_1,\dots,t_n) = (gt,t^{-1}t_1t,\dots,t^{-1}t_nt)=(gt,t_1,\dots,t_n).\]
Hence, $T$ acts trivially on itself. Thus the map $\theta_n$ 
factors through $(G\times T^n)/T=G\times_T T^n$. Thus there is a surjection 
$$\bar{\theta}_n: G/T\times T^n \to \Ho(\mathbb{Z}^n,G)_1.$$
Moreover, the Weyl group $W$ of $G$ acts diagonally on $G/T \times T^n$ by
\[w\cdot (gT,t_1,\dots,t_n) = (gwT,w^{-1}t_1w,\dots,w^{-1}t_nw),\]
where $gT$ is a coset in $G/T$. It follows that the map $\hat{\theta}_n$ is $W$-invariant since
\begin{align*}
(gw,w^{-1}t_1w,\dots,w^{-1}t_nw) & =((gw)w^{-1}t_1w(gw)^{-1},\dots,(gw)w^{-1}t_nw(gw)^{-1})\\
& =(gt_1 g^{-1},\dots,gt_n g^{-1}).
\end{align*}
Therefore, the map $\bar{\theta}_n$ factors through $G\times_{NT} T^n$ and so there are surjections
$$\hat{\theta}_n: G/T\times_W T^n \to \Ho(\mathbb{Z}^n,G)_{1_G}$$
for all $n$.

In what follows let $R=\Z\left[ 1/|W|\right]$ denote the ring of integers with the order of the Weyl group of $G$ inverted.

\begin{lemma}\label{fibre2}
If $G$ is a compact and connected Lie group with maximal torus $T$ and Weyl group $W$, then $H_{\ast}((\hat{\theta}_n)^{-1}(g_1,\dots,g_n);R)$ is isomorphic to the homology of a point $H_{\ast}(pt,R)$.
\end{lemma}
\begin{proof} 

See \cite[Lemma 3.2]{bairdcohomology} for a proof. Note that $\Q$ suffices for the ring of coefficients, but so does $R$.
\end{proof}
The following theorem is recorded to prove the next result. See \cite{bairdcohomology} for a proof.
\begin{thm}[Vietoris \& Begle]\label{viet-begle}
Let $h:X \longrightarrow Y$ be a closed surjection, where $X$ is a paracompact Hausdorff space. Suppose that for all $y \in Y$, $H_{\ast}(h^{-1}(y), R)= H_{\ast}(pt, R)$. Then the induced maps in homology
$h_{\ast}:H_{\ast}(X, R) \longrightarrow H_{\ast}(Y, R)$
are isomorphisms.
\end{thm}
Let $Comm(G)_{1_G}$ be the connected component of the trivial representation, see Definition  \ref{definition: G to Comm(G)}.

\begin{thm}\label{thm: first approximation}
Let $G$ be a compact and connected Lie group with maximal torus $T$ and Weyl group $W$. Then there
is an induced map $$\Theta: G \times_{NT} J(T) \to Comm(G)_{1_G}$$ together with a commutative diagram for all $n \geq 0$ given as follows:

\[
\begin{CD}
G \times T^m @>{\theta_m}>> \Ho(\Z^m,G)_{1_G} \\
 @VV{}V    @VV{}V       \\
G \times_{N_{T}}J(T)  @>{\Theta}>> Comm(G)_{1_G}.\\
\end{CD}
\]

Furthermore, $\Theta$ is a surjection and the homotopy theoretic fibre of this map $\Theta: G \times_{N_{T}}J(T) \to Comm(G)$
has reduced singular homology, which is entirely torsion of an order which divides the order of the Weyl group.
\end{thm}

\begin{proof}
There is an induced map
$$\Theta: G/T \times_{W} J(T) \to Comm(G)_{1_G},$$
which is a surjection.
Define $J_n(T)$ to be the $n$-th stage of the James construction defined by
the image of 
$$\bigsqcup_{0 \leq q \leq n}T^n$$ 
as a subspace of $J(T)$ with 
$$J(T) = {\colim} J_n(T)$$ 
for path-connected CW-complexes $T$. 
Next define $Comm_n(G)$ to be the image of 
$$\bigsqcup_{0 \leq q \leq n}\Ho(\mathbb Z^q,G)$$
as a subspace of $Comm(G)$ with  
$$Comm(G) = {\colim}Comm_n(G).$$
Then $\Theta$ restricts to maps
$$\Theta: G/T \times_{W} J_n(T) \to Comm_n(G)$$
for all integers $n\geq 1$.

Now consider the surjections
$$\hat{\theta}_n: G/T\times_W T^n \to \Ho(\mathbb{Z}^n,G)_{1_G}$$
which induce homology isomorphisms if $|W|$ is inverted. These maps restrict to surjections
$$\hat{\theta}_n: G/T\times_W S_n(T) \to S_n(G)$$
which similarly induce homology isomorphisms, where $S_n(T)$ is the set of all $n$-tuples in $T^n$ with at least one element the identity $1_G$. Note that $S(\Ho(\mathbb Z^n,T)) = S_n(T)$.

Moreover, the inclusions
$$ G/T\times_W S_n(T) \subset G/T\times_W T^n$$
and
$$S_n(G) \subset \Ho(\mathbb{Z}^n,G)_{1_G}$$
are cofibrations with cofibers $G/T\times_W \widehat{T}^n/(G/NT)$ and $\widehat{\Ho}(\Z^n,G)_{1_G}$, respectively. Equivalently, there is a commutative diagram of cofibrations

\[
\begin{CD}
G/T\times_W S_n(T) @>{i}>> G/T\times_W T^n @>>> G/T\times_W \widehat{T}^n/(G/NT)\\
 @VV{\hat{\theta}_n}V    @VV{\hat{\theta}_n}V     @VVV  \\
S_n(G)  @>{\Theta}>>  \Ho(\mathbb{Z}^n,G)_{1_G} @>>> \widehat{\Ho}(\Z^n,G)_{1_G}.\\
\end{CD}
\]

The five lemma applied to the long exact sequences in homology for the cofibrations shows that the maps
$$G/T\times_W \widehat{T}^n/(G/NT) \to \widehat{\Ho}(\Z^n,G)_{1_G} $$ 
induce isomorphisms in homology with the same coefficients. Note that from the stable decompositions in Theorems \ref{thm:first approximation INTRO} and \ref{thm: stable decompositions of X(q G)}, it follows that the cofibers above are the stable summands of the spaces $G/T \times_W J(T)$ and $Comm(G)_{1_G}$, respectively. The map 
$$\Theta: G/T \times_{W} J(T) \to Comm(G)_{1_G}$$
induces a map on the level of stable decompositions which is a homology isomorphism in each summand. Therefore, $\Theta$ induces a homology isomorphism with coefficients in $\Z[1/|W|]$.
\end{proof}

Note that this also proves Theorem \ref{thm: further.stable decompositions}.

\section{Proof of Theorem \ref{thm: homology of x2g}}\label{section: homology of x2g}

Let $G$ be a compact and connected Lie group and $R$ be the ring $\Z \left[{1}/{|W|}\right]$. Using Lemma \ref{fibre2} and Theorems \ref{viet-begle} and \ref{thm: first approximation}, the following theorem holds.
\begin{thm}\label{theorem: homology of GxJ(T)/NT}
Let $G$ be a compact and connected Lie group with maximal torus $T$ and Weyl group $W$. Then there is an isomorphism in homology
\[H_{\ast}(G \times_{NT} J(T); R ) \cong H_{\ast} ( Comm(G)_1; R ).\]
\end{thm}

\begin{thm}\label{theorem: homology of x2g1}
Let $G$ be a compact and connected Lie group with maximal torus $T$ and Weyl group $W$. Then there is an isomorphism in homology
\[H_{\ast}( Comm(G)_1; R) \cong \big( H_{\ast}(G/T;R) \otimes_{R}  \T[V]\big)_W .\]
\end{thm}
\begin{proof}
There is a short exact sequence of groups 
$$1 \to T \to NT \to W \to 1$$ 
and associated to it, there is a fibration sequence 
\[\big(G \times J(T)\big)/T \to \big(G \times J(T)\big)/NT \to BW,\] 
which is equivalent to the fibration
\[G \times_T J(T) \longrightarrow G \times_{NT} J(T) \longrightarrow BW.\]
The Leray spectral sequence has second page given by the groups
\[E^2_{p,q}=H_p(BW; H_q(G\times_T J(T); R))\]
which converges to $H_{p+q}(G\times_{NT} J(T); R)$. Since $|W|^{-1} \in R$, it follows that $E^2_{s>0,t}=0$ and the groups on the vertical axis are given by 
\[E^2_{0,t}=H_0\big( BW; H_t(G \times_{T} J(T);R)\big).\]
Recall that homology in degree 0 is given by the coinvariants
\[ H_0\big( BW; H_t(G \times_{T} J(T);R)\big) = \big(H_t(G \times_T J(T);R)\big)_W.\]
Also $T$ acts by conjugation and thus trivially on $T^n$, so it acts trivially on $J(T)$. Hence, $G \times_T J(T)=G/T \times J(T)$. 

The \textit{flag variety} $G/T$ has torsion free integer homology, see \cite{bott}, and so does $J(T)$. So the homology of $G/T \times J(T)$ with coefficients in $R$ is given by the following tensor product
\[ H_t(G \times_{T} J(T);R) =\bigoplus_{i+j=t} \big[ H_i(G/T;R) \otimes_{R}  H_j (J(T);R) \big].\]
The spectral sequence collapses at the $E^2$ term as stated above. Hence, 
\begin{align*}
H_0(BW;H_t(G \times_{T} J(T);R))\cong \bigg( \bigoplus_{i+j=t} \big[ H_i(G/T;R) \otimes_{R}  H_j (J(T);R) \big]\bigg)_W.
\end{align*}
Using Theorem \ref{theorem: homology of GxJ(T)/NT}, it follows that 
\[H_{\ast} ( Comm(G)_1; R ) \cong H_{\ast}(G \times_{NT} J(T);R) = \big( H_{\ast}(G/T;R) \otimes_{R}  H_{\ast} (J(T);R)\big)_W.\]
Recall that the homology of $J(T)$ is the tensor algebra on the reduced homology of $T$. Let $\T[V]$ denote the tensor algebra on the reduced homology of $T$, denoted by $V$. Then there is an isomorphism
\[ H_{\ast} ( Comm(G)_1; R ) = \big( H_{\ast}(G/T;R) \otimes_{R}  \T[V]\big)_W.\]
\end{proof}

If $G$ has the property that every abelian subgroup is contained in a path-connected abelian subgroup, then $\Ho(\Z^n,G)$ is path-connected, see \cite{adem2007commuting}. Some of these groups include $U(n)$, $SU(n)$ and $\Sp(n)$. In this case $Comm(G)$ is also path-connected and it follows that  $Comm(G)_1=Comm(G)$.

\begin{cor}
Let $G$ be a compact, connected Lie group with maximal torus $T$ and Weyl group $W$, such that  every abelian subgroup is contained in a path-connected abelian subgroup. Then there is an isomorphism in homology
\[H_{\ast}( Comm(G); R) \cong \big( H_{\ast}(G/T;R) \otimes_{R}  \T[V]\big)_W .\]
\end{cor}

To find the homology groups of $Comm(G)_1$ explicitly, it is necessary to find the coinvariants
\[\big( H_{\ast}(G/T;R) \otimes_{R}  \T[V]\big)_W .\]
This leads the subject to representation theory. Note that Theorem \ref{theorem: homology of x2g1} can also be used to study the cases of the compact and connected simple exceptional Lie groups $G_2,F_4,E_6,E_7$ and $E_8$.

\section{Proof of Theorem \ref{theorem:ungraded homology of X2G INTRO}}\label{section: ungraded homology of X2G}

Let $H_{\ast}^U$ denote ungraded homology and $\T_U [V]$ denote the ungraded tensor algebra over $V$, where $V$ is the reduced homology of $T$ with coefficients in $R$.

The homology $H_{\ast}(G/T;R)$, if considered ungraded, is isomorphic as an $R[W]$-module to the group ring of $W$, namely $R[W]$. This fact was proven in \cite[Proposition B.1]{bairdcohomology}, thus ignoring  the grading of the homology $H_{\ast}(G/T;R)$ in the proof of Theorem \ref{theorem: homology of x2g1}, as a $W$-module the homology is isomorphic to $R[W]$. Furthermore, as an ungraded module, there is an isomorphism 
$$H_*(G/T) \otimes_{R[W]} \T [V] \to  \T[V].$$ 
The next theorem follows.

\begin{thm}\label{theorem:ungraded homology of X2G}
Let $G$ be a compact, connected Lie group with maximal torus $T$ and Weyl group $W$. Then there is an isomorphism in ungraded homology
\[H_{\ast}^U( Comm(G)_1; R) \cong  \T_U [V] .\]
\end{thm}
\begin{proof}
From Theorem \ref{theorem: homology of x2g1} there is an isomorphism in homology given by
\[H_{\ast}( Comm(G)_1; R) \cong \big( H_{\ast}(G/T;R) \otimes_{R}  \T[V]\big)_W .\]
If all homology is ungraded, then there are isomorphisms in ungraded homology given by
\[
H_{\ast}^U( Comm(G)_1; R)  \cong \big( R[W] \otimes_{R}  \T_U[V]\big)_W \cong R[W] \otimes_{R[W]}  \T_U[V]   \cong \T_U[V].
\]
\end{proof}
This shows that as an abelian group, without the grading, the homology of $Comm(G)_1$ with coefficients in $R$ is the ungraded tensor algebra $\T_U[V]$. The following is an immediate corollary of Theorem \ref{theorem:ungraded homology of X2G}.

\begin{cor}
Let $G$ be a compact, connected Lie group with maximal torus $T$ and Weyl group $W$, such that  every abelian subgroup is contained in a path-connected abelian subgroup. Then there is an isomorphism in ungraded homology
\[H_{\ast}^U( Comm(G); R) \cong  \T_U [V] .\]
\end{cor}

\section{An example given by $SO(3)$}\label{section: example SO(3)}

Consider the special orthogonal group $SO(3)$. The connected components of the space $Comm(SO(3))$
are given next. D. Sjerve and E. Torres-Giese \cite{giese.sjerve} showed that the space $\Ho(\Z^n,SO(3))$ has the following decomposition into path components
$$
\Ho(\Z^n,SO(3)) \approx \Ho(\Z^n,SO(3))_1 \bigsqcup \bigg( \bigsqcup_{N_n}  S^3/Q_8 \bigg),
$$
where $Q_8$ is the group of quaternions acting on the 3-sphere. $N_n$ is a finite positive integer depending on $n$ and equals 
$\frac{1}{6}\left(4^n+3\cdot 2^n +2 \right)$ if $n$ is even, and 
$\frac{2}{3}(4^{n-1}-1)-2^{n-1}+1$ otherwise. By definition it follows that 
$$
Comm(SO(3)) = \bigg(\bigsqcup_{n \geq 1} \big[ \Ho(\Z^n,SO(3))_1 \bigsqcup \big( \bigsqcup_{N_n}  S^3/Q_8 \big) \big] \bigg)/\sim,
$$
where $\sim$ is the relation in Definition \ref{definition: spaces of all commuting n-tuples}.

An $n$-tuple $(M_1,\dots , M_n)$ is in the connected component $\Ho(\Z^n,SO(3))_1$ if and only if all the matrices $M_1,\dots , M_n$ are rotations about the same axis. The $n$-tuple is in one of the components homeomorphic to $S^3/Q_8$ if and only if there are two matrices $M_i$ and $M_j$ that are rotations about orthogonal axes such that the other coordinates are equal to one of $M_i$, $M_j$, $M_i M_j$ or $1_G$, see \cite{giese.sjerve}. For any positive integer $n$ there is at least one $n$-tuple not in $\Ho(\Z^n,SO(3))_1$ with no coordinate the identity. Therefore, in the identifications above the number of copies of $S^3/Q_8$ increases with $n$ and there are infinitely many copies of $S^3/Q_8$ as connected components of $Comm(SO(3))$. So it follows that $Comm(SO(3))$ has the following path components
$$Comm(SO(3)) =  Comm(SO(3))_1 \bigsqcup \big( \bigsqcup_{\infty}  S^3/Q_8 \big).$$
See \cite{stafa.thesis} for more details. Note that Theorem \ref{thm:souped-up-Molien INTRO} gives information about the rational cohomology as well as the integral cohomology with $2$ inverted of $Comm(SO(3))_1$.

\section{The cases of $U(n)$, $SU(n)$, $Sp(n)$, and $Spin(n)$ }\label{section:Un.SUn}

Recall the map of spaces
$$\Theta: G\times_{NT}J(T) \to Comm(G)_{1},$$ 
which induces a map in singular homology which is an isomorphism under the conditions that $G$ is compact, simply-connected, and the order of the Weyl group $W$ for a maximal torus $T$ has been inverted in the coefficient ring (Theorem \ref{theorem: homology of x2g1}). The purpose of this section is to describe the ungraded homology of $Comm(G)$ for $G$ one of the groups $U(n)$, $SU(n)$, $Sp(n)$, and $Spin(n)$.

A maximal torus for $U(n)$ is of rank $n$ given by $$T = (S^1)^n.$$ Thus the ungraded homology of $Comm(U(n))$ has Poincar\'e series 
$$ \Hilb(Comm(U(n)),t)= 1/(2-(1+t)^n).$$ The analogous result for $SU(n)$ follows, that is, the Poincar\'e series is given by
$$\Hilb(Comm(SU(n)),t)=1/(2-(1+t)^{n-1}).$$
These series give the ungraded homology of $Comm(G)$ for $G = U(n)$ and $SU(n)$, respectively.

The case of $Sp(n)$ is analogous as a maximal torus is of rank $n$, so the Poincar\'e series for the ungraded 
homology of $Comm(Sp(n))$ with $n!$ inverted is 
$$\Hilb(Comm(Sp(n)),t)=1/(2-(1+t)^n).$$

The cases of $SO(n)$ and $Spin(n)$ breaks down classically into two cases as expected with $n = 2a \  \mbox{or} \ n = 2b+1$. 

First, the case of $$n = 2b+1.$$ The rank of a maximal torus (finite product of circles) is $b$. Thus the Poincar\'e series for the ungraded homology of $Comm(Spin(2b+1),t)$ with $n! = (2b+1)!$ inverted is 
$$\Hilb(Comm(Spin(2b+1)),t)=1/(2-(1+t)^b).$$

Second, the case of $$n = 2a.$$ The rank of a maximal torus (finite product of circles) is $a$. Thus the Poincar\'e series for the ungraded homology of $Comm(Spin(2a),t)$ with $n! = (2a)!$ inverted is 
$$\Hilb(Comm(Spin(2a)),t)=1/(2-(1+t)^a).$$


\section{Results for $G_2$, $F_4$, $E_6$, $E_7$ and $E_8$}\label{section:exceptional Lie groups}

In this section applications of earlier results concerning $Comm(G)$ are applied to the classical exceptional, simply-connected, simple compact Lie groups $G_2$, $F_4$, $E_6$, $E_7$ and $E_8$. In particular, the ungraded homology of $Comm(G)$ is given in these cases.

Recall the map 
$$\Theta: G\times_{NT}J(T) \to Comm(G).$$ 
In the special cases of this section, if the rank of the maximal torus is $r$, then a choice of maximal torus will be denoted $$T_r.$$ Thus the following is classical, see \cite{chevalley1999theory}. 
The first explicit determination of the Poincar\'e  polynomials of the exceptional simple Lie groups has been accomplished by Yen Chih-Ta \cite{yen.chih.ta}.

\begin{enumerate}
\item If $G = G_2$, then $r=2$, the Weyl group $W$ is the dihedral group of order $12$, and 
the cohomology with $6$ inverted is an exterior algebra on classes in degrees $3$, and $11$,
\item If $G = F_4$, then $r=4$, the Weyl group $W$ is the dihedral group of order $2^7\times 3^2=1,152$, and
and  the cohomology with $6$ inverted is an exterior algebra on classes in degrees $3,11,15,23$.
\item If $G = E_6$, then $r=6$, the Weyl group $W$ is $O(6, \mathbb F_2)$ of order $51,840$, and
and  the cohomology with $6$ inverted is an exterior algebra on classes in degrees $3,9,11,15,17,23$.
\item If $G = E_7$, then $r=7$, the Weyl group $W$ is $O(7,\mathbb F_2) \times \mathbb Z/2$ of order $2,903,040$, and
and the cohomology with $2\cdot3\cdot 5 \cdot 7$ inverted is an exterior algebra on classes in degrees $3, 11, 15, 19, 23, 27, 35$.
\item If $G = E_8$, then $r=8$, the Weyl group $W$ is a double cover of  $O(8,\mathbb F_2)$ of order $(2 ^{14})(3^5)(5^2)(7)$, and  the cohomology with $2 \cdot 3 \cdot 5\cdot 7$ inverted is an exterior algebra on classes in degrees $3,15,23,27,35,39, 47, 59$, and $11$,
\end{enumerate}

Finally, recall Theorem \ref{theorem:ungraded homology of X2G} restated as follows where
$H_{\ast}^U(X; R)$ denotes ungraded homology.
\begin{thm}
Let $G$ be one of the exceptional Lie groups $G_2, F_4, E_6, E_7, E_8$ with maximal torus $T$ and Weyl group $W$. Then there is an isomorphism in ungraded homology
\[H_{\ast}^U( Comm(G); R) \cong  \T_U [V] .\]
\end{thm}

One consequence is that if homology groups are regarded as ungraded, then the copies of the group ring can be canceled, and this gives the ungraded homology of $Comm(G)$ in the cases of these exceptional Lie groups, which are recorded as follows.

\begin{cor}\label{cor: rational.for.exceptional simple Lie groups}
Let $G$ be one of $G_2, F_4, E_6, E_7, E_8$. The ungraded homology of the space $Comm(G)_{1}$ 
in these cases with coefficients in $\mathbb Z[1/|W|]$, where $|W|$ is the order of the Weyl group, is given as follows: The ungraded homology is isomorphic to $$\T[\widetilde{H}_*(T)]$$
where $T$ is a maximal torus. For each of the groups, the Hilbert-Poincar\'e series is given as follows
\begin{enumerate}
\item If $G = G_2$, then the ungraded homology of $Comm(G)_1$
is the tensor algebra $\T[\widetilde{H}_*((S^1)^2)]$, and has Hilbert-Poincar\'e series $$1/(1-2t-t^2).$$

\item If $G = F_4$, then the ungraded homology of $Comm(G)_1$
is the tensor algebra $\T[\widetilde{H}_*((S^1)^4)]$, and has Hilbert-Poincar\'e series $$1/(2-(1+t)^4).$$

\item If $G = E_6$, then the ungraded homology of $Comm(G)_1$
is the tensor algebra $\T[\widetilde{H}_*((S^1)^6)]$, and has Hilbert-Poincar\'e series $$1/(2-(1+t)^6).$$

\item If $G = E_7$, then the ungraded homology of $Comm(G)_1$
is the tensor algebra $\T[\widetilde{H}_*((S^1)^7)]$, and has Hilbert-Poincar\'e series $$1/(2-(1+t)^7).$$

\item If $G = E_8$, then the ungraded homology of $Comm(G)_1$
is the tensor algebra $\T[\widetilde{H}_*((S^1)^8)]$, and has Hilbert-Poincar\'e series $$1/(2-(1+t)^8).$$
\end{enumerate}
\end{cor}

\appendix

\section{Proof of Theorem \ref{thm:souped-up-Molien INTRO}}\label{appendix: proof of proposition}

\begin{center}
{V. Reiner
}\end{center}

Let $W$ be a finite subgroup of $GL_n(\RR)$ generated by reflections acting on $\RR^n$. Then $W$ also acts in a grade-preserving fashion 
on the polynomial algebra 
$$
R=\RR[x_1,\ldots,x_n],
$$ 
where $x_1,\ldots,x_n$ are a basis for the dual space $(\RR^n)^*$.
The theorem of Shephard-Todd and Chevalley \cite[\S 4.1]{broue.reflextion.gps}
asserts that the subalgebra of $W$-invariant polynomials is again a polynomial algebra
$$
R^W=\RR[f_1,\ldots,f_n].
$$  
One can choose the $f_1,\ldots,f_n$ homogeneous, say with degrees
$d_1,\ldots,d_n$.  For example, when $W$ is the symmetric group $S_n$ permuting coordinates in $\RR^n$, 
one can choose $f_i=e_i(x_1,\ldots,x_n)$ 
the {\it elementary symmetric functions}, 
and one has $(d_1,\ldots,d_n)=(1,2,\ldots,n)$.

The usual grading conventions for the cohomology of a topological space requires that all degrees $d_j$ be doubled in the formulas here.  For example, in the case of $G= U(n)$ with $W$ given by the symmetric group on $n$-letters, the value $d_j = j$,  but the homological degree of the $j$-th Chern class is degree $2j$.

Then $R^W$ has the following $\NN$-graded Hilbert series in the variable $q$:
\begin{equation}
\label{polynomial-invariants-hilb}
\Hilb(R^W,q):=\sum_{i=0}^\infty \dim_\RR(R^W_i)\,\, q^{i} 
=\prod_{j=1}^n \frac{1}{1-q^{2d_j}}.
\end{equation}

The {\it coinvariant algebra} is the quotient ring 
$$
C:=R/(f_1,\ldots,f_n)=R/(R^W_+),
$$
which also carries a grade-preserving $W$-action.  

In addition, we consider the $W$-action on the
exterior algebra $E=\wedge \RR^n$,
the reduced exterior algebra $\barE:=\bigoplus_{k=1}^n \wedge^k \RR^n$,
and on the $\RR$-dual $\TTT^*[\barE]$ of the tensor 
algebra over $\RR$ on $\barE$.
These will have their own separate $\NN^2$-bi-grading 
so that the $\RR$-dual of $\wedge^k \RR^n$ lies in bidegree $(k,1)$,
and the $\RR$-dual of
$\wedge^{k_1} \RR^n \otimes \cdots \otimes \wedge^{k_m} \RR^n$ lies in
the bidegree $(k_1+\cdots+k_m,m)$ of $\TTT^*[\barE]$.
Thus the diagonal action of $W$ 
on the tensor product $C \otimes \TTT^*[\barE]$ over $\RR$
respects an {\it $\NN^3$-tri-grading}.  Its $W$-fixed space
has trigraded Hilbert series defined by
$$
\begin{aligned}
&\Hilb\left( \left(C \otimes \TTT^*[\barE] \right)^W , q, s, t \right) \\
&\quad := \sum_{i,m=0}^\infty \sum_{\substack{(k_1,\ldots,k_m) \\ \text{in }\{1,2,\ldots\}^m} }
     \dim_\RR \left( C^i \otimes 
      \left(\wedge^{k_1} \RR^n \otimes \cdots \otimes \wedge^{k_m} \RR^n\right)^* 
         \right)^W  
q^{i} s^{k_1+\cdots+k_m} t^m.
\end{aligned}
$$

In the next theorem, the exponent in $q^{2i}$ is doubled as these degrees correspond to the topological degrees in the invariant algebra.

\begin{thm}\label{thm:souped-up-Molien}
If $G$ is a compact, connected Lie group with maximal torus $T$, and Weyl group $W$, then
$$
\begin{aligned}
&\Hilb\left( \left(C^* \otimes \TTT^*[\barE] \right)^W , q, s, t \right) \\
&\quad =\frac{ \prod_{i=1}^n (1-q^{2d_i}) }{ |W| }\sum_{w \in W} \frac{1}{\det(1-q^2w) \left( 1-t(\det(1+sw)-1) \right)}.
\end{aligned}
$$

\end{thm}
\begin{proof}
For any (ungraded) finite-dimensional $W$-representation $X$ over $\RR$, 
one has an isomorphism \cite[(4.5)]{broue.reflextion.gps} of $\NN$-graded $\RR$-vector spaces 
\begin{equation}
\label{Broue-isomorphism}
\left( R \otimes X \right)^W
\cong R^W \otimes \left( C^* \otimes X \right)^W.
\end{equation}
If the $W$-action respects an additional $\NN^2$-grading on $X=\oplus_{(j,m)} X_{j,m}$, 
separate from the $\NN$-grading on $R$, then 
\eqref{Broue-isomorphism} becomes an
isomorphism of $\NN^3$-trigraded $\RR$-vector spaces.  Consequently, one has
\begin{equation}
\label{tensor-consequence}
\begin{aligned}
\Hilb\left( \left( C^* \otimes X \right)^W, q,s,t\right) 
&=
\frac{ 
\Hilb\left( \left( R \otimes X \right)^W, q,s,t\right) 
}
{ 
\Hilb( R^W, q) 
}\\
&=
\prod_{i=1}^n (1-q^{2d_i}) \cdot
\Hilb\left(( \left( R \otimes X \right)^W, q,s,t\right) 
\end{aligned}
\end{equation}
using \eqref{polynomial-invariants-hilb} for the last equality.

Next note that 
\begin{equation}
\label{bigraded-trace-consequence}
\begin{aligned}
\Hilb\left( \left( R \otimes X \right)^W, q,s,t\right)
&=\sum_{i,j,m = 0}^{\infty}
  \dim_\RR \left( R_i \otimes X_{j,m} \right)^W  q^{i} s^j t^m \\
&=\frac{1}{|W|} \sum_{w \in W}
\sum_{i,j,m = 0}^{\infty}
  \Trace\left(w|_{R_i \otimes X_{j,m}} \right) q^{i} s^j t^m
\end{aligned}
\end{equation}
as $\pi: v \longmapsto 1/|W| \sum_{w \in W} w(v)$ is
an idempotent projection onto the $W$-fixed subspace $V^W$ of any
$W$-representation $V$, so 
the trace of $\pi$ is $\dim_\RR(V^W)$.
Taking $X=\TTT^*[\barE]$ in \eqref{tensor-consequence} and \eqref{bigraded-trace-consequence},  the theorem follows
from this claim: for any $w$  in $W$,
\begin{equation}
\label{trace-claim}
\sum_{i,j,m = 0}^{\infty}
  \Trace\left(w|_{R_i \otimes \TTT^*[\barE]_{j,m}} \right) q^{i} s^j t^m =
    \frac{1}{\det(1-q^2w) \left( 1-t(\det(1+sw)-1) \right)}
\end{equation}
To prove the claim \eqref{trace-claim}, start with the facts \cite[Example 3.25]{broue.reflextion.gps} that 
\begin{eqnarray}
\label{polynomials-graded-trace}
\sum_{i=0}^\infty 
  \Trace\left(w|_{R_i} \right) q^{i}  
\displaystyle
&=& \frac{1}{\det(1-q^2w^{-1})}
= \frac{1}{\det(1-q^2w)}, \\
\label{exterior-graded-trace}
\sum_{j=0}^n
  \Trace\left(w|_{\wedge^j \RR^n} \right) s^j  
&=&\det(1+sw)=\det(1+sw^{-1}).
\end{eqnarray}
The rightmost equalities in 
\eqref{polynomials-graded-trace},\eqref{exterior-graded-trace}
arise because $w$ acts orthogonally on $\RR^n$,
forcing $w, w^{-1}$ to have the same eigenvalues with multiplicities, 
From \eqref{exterior-graded-trace} one deduces that
$$
\sum_{j=1}^n
  \Trace\left(w|_{\wedge^j \RR^n} \right) s^j 
= \det(1+sw)-1.
$$
Since (the dual of) $\wedge^j \RR^n$ 
lies in the bidegree $\TTT^*[\barE]_{j,1}$, 
this implies
\begin{equation}
\label{tensors-graded-trace}
\begin{aligned}
\sum_{j,m=0}^\infty
  \Trace\left( w|_{\TTT^*[\barE]_{j,m}} \right) s^j t^m  
&= \frac{1}{1-t\left(\det(1+sw^{-1})-1\right)} \\
&= \frac{1}{1-t\left(\det(1+sw)-1\right)}.
\end{aligned}
\end{equation}
Consequently the claim \eqref{trace-claim} follows from
\eqref{polynomials-graded-trace}
and \eqref{tensors-graded-trace}, since
graded traces are multiplicative on
graded tensor products.  This completes the proof.
\end{proof}

\section{acknowledgments}
The authors thank V. Reiner for including the appendix in this paper along with his formula giving the Hilbert-Poincar\'e series for $Comm(G)$.

\end{document}